\documentclass{amsart}
\usepackage{amssymb, latexsym, amsthm, enumitem, color, amsmath, tikz-cd, bbm, mathtools}

\usepackage{array}
\usepackage{blindtext}
\usepackage{longtable}
\usepackage{enumitem}
\usepackage{multirow}
\usepackage{multicol}
\usepackage{centernot}
\usepackage{bbm}
\usepackage{nicefrac}
\usepackage{upquote}

\usepackage[unicode=true,pdfusetitle,bookmarks=true,bookmarksnumbered=false,
bookmarksopen=false,breaklinks=true,pdfborder={0 0 0},
pdfborderstyle={},backref=false,colorlinks=true]{hyperref}
\definecolor{myblue}{rgb}{0.09,0.32,0.44}
\hypersetup{pdfborder={0 0 0},pdfborderstyle={},colorlinks=true,linkcolor=myblue,citecolor=myblue,urlcolor=blue}

\newtheorem{thm}{Theorem}[section]
\newtheorem*{thm*}{Theorem}

\newtheorem{conj}[thm]{Conjecture}
\newtheorem{cor}[thm]{Corollary}

\newtheorem{exmpl}[thm]{Example}

\newtheorem{lem}[thm]{Lemma}
\newtheorem{prop}[thm]{Proposition}

\newtheorem{obs}[thm]{Observation}

\theoremstyle{remark}
\newtheorem{rem}[thm]{Remark}

\newtheorem*{rem*}{Remark}
\newtheorem*{rems*}{Remarks}

\newcommand\Cref[1]{{Corollary~\ref{#1}}}
\newcommand\Cjref[1]{{Conjecture~\ref{#1}}}

\newcommand\Lref[1]{{Lemma~\ref{#1}}}
\newcommand\Pref[1]{{Proposition~\ref{#1}}}
\newcommand\Rref[1]{{Remark~\ref{#1}}}
\newcommand\Tref[1]{{Theorem~\ref{#1}}}
\newcommand\Sref[1]{{Section~\ref{#1}}}

\newcommand{\Z}{\mathbb{Z}}
\newcommand{\set}[1]{\left\{#1\right\}}
\newcommand{\sub}{\subseteq}
\newcommand{\floor}[1]{\left\lfloor #1 \right\rfloor}
\newcommand{\ceil}[1]{\left\lceil #1 \right\rceil}
\newcommand{\suchthat}{\,\ifnum\currentgrouptype=16\middle\fi|\,}
\newcommand{\id}{\mathsf{id}}
\renewcommand{\Pr}{\mathbb{P}}
\newcommand{\ind}[1]{\mathbbm{1}_{#1}}
\newcommand{\range}{\operatorname{range}}

\newcommand{\E}{\mathbb{E}}
\newcommand{\eps}{\varepsilon}
\newcommand{\Cay}{\mathrm{Cay}}
\newcommand{\dist}{\operatorname{dist}}
\newcommand{\Reff}{\mathcal{R}_{\text{eff}}}
\newcommand{\supp}{\operatorname{supp}}

\makeatletter
\def\moverlay{\mathpalette\mov@rlay}
\def\mov@rlay#1#2{\leavevmode\vtop{
   \baselineskip\z@skip \lineskiplimit-\maxdimen
   \ialign{\hfil$\m@th#1##$\hfil\cr#2\crcr}}}
\newcommand{\charfusion}[3][\mathord]{
    #1{\ifx#1\mathop\vphantom{#2}\fi
        \mathpalette\mov@rlay{#2\cr#3}
      }
    \ifx#1\mathop\expandafter\displaylimits\fi}
\makeatother

\title[Recurrence of stationary random walks on lamplighter groups]{Phase transition for recurrence of stationary random walks on lamplighter groups}
\author{Itai Benjamini}
\address{Department of Mathematics, The Weizmann Institute of Science, Rehovot}
\email{itai.benjamini@weizmann.ac.il}
\author{Guy Blachar}
\address{Department of Mathematics, The Weizmann Institute of Science, Rehovot}
\email{guy.blachar@gmail.com}
\author{Ariel Yadin}
\address{Department of Mathematics, Ben-Gurion University of the Negev, Beer-Sheva}
\email{yadina@bgu.ac.il}

\begin{document}

\begin{abstract}
  We introduce and study a class of random walks on lamplighter groups $H\wr G$, where $H$ is a nontrivial finitely generated group and $G$ is an infinite finitely generated group, called \textbf{stationary random walks}. At each step, the walk switches the lamp at its current position, moves in the base group with a drift towards the identity, and switches the lamp again at the new position. We show that when $G$ is virtually-$\Z$ and $H$ is finite, these walks exhibit a phase transition between recurrence and transience, while when~$G$ is not virtually-$\Z$ or $H$ is infinite, they are always transient. In the case $G=\Z$, we determine the exact critical parameter and provide a quantitative description of this phase transition.
\end{abstract}

\maketitle

\section{Introduction}

Consider the lamplighter group $F\wr\Z$, where $F$ is a nontrivial finite group. Elements of this group can be described as pairs $((L(i))_{i\in\Z},x)$, where $(L(i))_{i\in\Z}\in\bigoplus_{\Z}F$ and $x\in\Z$. One may think of $(L(i))_{i\in\Z}$ as a lamp configuration on~$\Z$, in which an $F$-valued lamp is placed at each point of~$\Z$, but only finitely many are allowed to be nontrivial. The element $x\in\Z$ represents the position of a lamplighter, who walks along the $\Z$-axis and can switch the lamps' values. Such groups play a central role in geometric group theory, particularly in the study of random walks on groups.

This paper studies a family of random walks $R_n$ on the group~$F\wr\Z$, which we refer to as \textbf{stationary random walks}. For a fixed parameter $\frac{1}{2}<p<1$, we define~$R_n$ to be a specific switch-walk-switch random walk on the group: at each step, the walk first refreshes the lamp at its current location on the $\Z$-axis (``switch''), then moves along the $\Z$-axis (``walk''), and then refreshes the lamp at its new location (``switch''). We assume the switches are uniform on $F$, while the walk on $\Z$ has a constant drift $p$ towards the origin.

Formally, the projection of $R_n$ on $\Z$, denoted $S_n$, is given by
\[
    \Pr(S_{n+1}=x+1\mid S_n=x) = \begin{cases} p, & x<0 \\ \frac{1}{2}, & x=0 \\ 1-p, & x>0 \end{cases}
\]
and
\[
    \Pr(S_{n+1}=x-1\mid S_n=x) = \begin{cases} 1-p, & x<0 \\ \frac{1}{2}, & x=0 \\ p, & x>0. \end{cases}
\]

The corresponding random walk $R_n$ on the lamplighter group is then defined as follows: Let $(S_n)_n$ denote the above random walk on $\Z$, and let $(U_n,V_n)_n$ be i.i.d.\ random variables uniformly distributed on $F$. Define $R_n = ((L_n(i))_{i \in \Z} , S_n)$, where $L_{n+1}$ is obtained from $L_n$ via
\begin{equation}
\label{eqn:L dfn}
L_{n+1}(i) =
\begin{cases}
L_n(i) & \textrm{ if } i \not\in \{ S_n , S_{n+1} \} , \\
U_{n+1} & \textrm{ if } i=S_n , \\
V_{n+1} & \textrm{ if } i=S_{n+1} .
\end{cases}
\end{equation}

We study the asymptotic behaviour of the walk depending on the value of the parameter $p$. In particular, regarding the recurrence or transience of $R_n$, we obtain the following phase transition:

\begin{thm}\label{thm:Z-phase}
  Let $R_n$ denote the stationary random walk on $F\wr\Z$ with parameter~$p$. Then $R_n$ is recurrent if $p\ge\frac{\left|F\right|^2}{\left|F\right|^2+1}$, and is transient if $p<\frac{\left|F\right|^2}{\left|F\right|^2+1}$.
\end{thm}

We give a quantitative version of the above theorem by considering the local time of the walk. We denote the local time of $R_n$ at a point $a\in F\wr\Z$ until time~$n$~by
\[
    \xi(a,n) = \left|\set{0\le k\le n\suchthat R_k=a}\right|.
\]
For two sequences $(\alpha_n)_n,(\beta_n)_n$ of positive real numbers, we write $\alpha_n\lesssim \beta_n$ if there exists a constant $C>0$ such that $\alpha_n\le C\beta_n$ for all $n$, and we write $\alpha_n\simeq \beta_n$ if $\alpha_n\lesssim \beta_n$ and $\beta_n\lesssim \alpha_n$. We show:

\begin{thm}\label{thm:Z-local-time}
  Let $R_n$ denote the stationary random walk on $F\wr\Z$ with parameter~$p$. Then
  \begin{equation}\label{eq:local-time}
    \E[\xi(\id,n)] \simeq \begin{cases}1, & p<\frac{\left|F\right|^2}{\left|F\right|^2+1} \\ \log n, & p=\frac{\left|F\right|^2}{\left|F\right|^2+1} \\ n^{1-2\alpha}, & p>\frac{\left|F\right|^2}{\left|F\right|^2+1}\end{cases}
  \end{equation}
  where $\alpha=\frac{\log\left|F\right|}{\log\frac{p}{1-p}}$, and $\id$ denotes the identity element of $F\wr\Z$. (The implied constants of \eqref{eq:local-time} depend on $p$ and $\left|F\right|$.)
\end{thm}

It would also be interesting to estimate other quantitative aspects of stationary random walks, such as the expected range of $R_n$. We expect it to grow polynomially in $n$, with an exponent that depends on $p$ and equals $1$ when $p<\frac{\left|F\right|^2}{\left|F\right|^2+1}$. However, we were not able to establish concrete bounds.\medskip

We further study whether this phenomenon holds when replacing $\Z$ by another finitely generated group $G$, and $F$ by an arbitrary finitely generated group $H$. In this setting, we define the stationary random walk on~$H\wr G$ analogously, where the walk on $G$ is the \textit{$\lambda$-homesick random walk}. Homesick random walks are a notion of random walks on groups with drift towards the identity (or towards the root of a rooted graph), where the parameter $\lambda\ge 1$ controls the drift of the walk -- $\lambda=1$ corresponds to the simple random walk, and the larger~$\lambda$ is, the stronger the drift (see precise definition in \Sref{sec:homesick}). For $G=\Z$ with the standard generating set $S=\set{\pm 1}$, the $\lambda$-homesick random walk reduces to the biased walk on $\Z$ with constant drift $p=\frac{\lambda}{\lambda+1}$ towards $0$, as above.

We show that the above phase transition extends to the case where the group $G$ is virtually-$\Z$, i.e.:

\begin{thm}\label{thm:virt-Z}
  Let $G$ be a virtually-$\Z$ group, and let $H$ be a nontrivial finite group. Then there exists $1<\lambda_c<\infty$ such that the $\lambda$-stationary random walk on $F\wr G$ is transient if $\lambda<\lambda_c$, and recurrent if $\lambda>\lambda_c$. Moreover, the value of $\lambda_c$ depends on~$H$ only through $|H|$.
\end{thm}

In fact, the theorem also holds when the switches are made according to some finitely supported, symmetric and non-degenerate random walk, and in this case the value of $\lambda_c$ does not change. We conjecture that the critical behaviour is like the case $G=\Z$:

\begin{conj}\label{conj:conj}
  If $G$ is virtually-$\Z$ and $H$ is a nontrivial finite group, then the $\lambda$-stationary random walk on $H\wr G$ is recurrent also for $\lambda=\lambda_c$.
\end{conj}

Finally, we show that the phase transition does not occur in any other case:

\begin{thm}\label{thm:no-transition}
  Let $G$ and $H$ be finitely generated groups, where $G$ is infinite and~$H$ is nontrivial. If $G$ is not virtually-$\Z$, or if $H$ is infinite (and the switches are made according to some finitely supported, symmetric and non-degenerate random walk), then the $\lambda$-stationary random walk on $H\wr G$ is transient for every $\lambda\ge 1$.
\end{thm}

\begin{rem}
  The notion of stationary random walks extends naturally to lamplighter graphs $\Gamma' \wr \Gamma$. We discuss this in \Sref{sec:graphs}.
\end{rem}

\begin{rem}
  Similar phenomena to those described above also hold for the homesick random walks themselves. Lyons \cite{Lyons95} showed that a $\lambda$-homesick random walk $R_n$ on a Cayley graph $\Cay(G,S)$ is transient if $1\le\lambda<\rho(G,S)\coloneqq\lim_{n\to\infty}\left|S^n\right|^{1/n}$, and is recurrent if $\lambda>\rho(G,S)$. Furthermore, Lyons--Pemantle--Peres~\cite{LyonsPemantlePeres96} and Revelle~\cite{Revelle01} studied the $\lambda$-homesick random walk on the classical lamplighter group $(\Z/m\Z)\wr\Z$, showing that when the walk is transient it has positive speed (even though the speed of the simple random walk on $(\Z/m\Z)\wr\Z$ is zero).

  In contrast, the bias of the stationary random walks on $F\wr G$ is not directed towards a specific vertex of the Cayley graph, but depends only on the base group~$G$. Moreover, the recurrence--transience phase transition occurs only when $G$ is one-dimensional, whereas Lyons' result applies to any group.
\end{rem}

\subsection*{Main ideas and outline of the paper}

Consider the $\lambda$-stationary random walk~$R_n$ on $H\wr G$, where $H$ is a finite group, and the switch measure is uniform on $H$. The phase transition between transience and recurrence stems from the fact that conditioned on the path of the projection $Z_n$ of $R_n$ on~$G$, the values of the lamps along the path are independent and uniform on $H$. Conditioning further on $Z_n=\id_G$, we have
\begin{equation}\label{eq:ret-prob-idea}
  \Pr(R_n=\id\suchthat Z_1, ..., Z_{n-1}, Z_n=\id_G) = \left|H\right|^{-\left|\range(Z_n)\right|},
\end{equation}
where $\range(Z_n)=\set{Z_0,\dots,Z_n}$.

In general, when the random walk $Z_n$ is recurrent, we expect that $\range(Z_n)$ will be close to a ball of radius $C\log n$ in the Cayley graph of $G$ (where $C$ depends on $\lambda$). When $G$ is virtually-$\Z$, the size of a ball in its Cayley graph is linear in its radius. The constant $C$ decreases as $\lambda$ grows, so when $\lambda$ is sufficiently close to $1$ the probabilities \eqref{eq:ret-prob-idea} will be summable, whereas when $\lambda$ is sufficiently large they will not be summable. This gives the claimed phase transition. When $G$ is not virtually-$\Z$, the size of balls in the Cayley graph grows at least quadratically in the radius, and thus \eqref{eq:ret-prob-idea} is summable for all $\lambda>1$, so the walk is transient for every value of $\lambda>1$. When $H$ is infinite, the probability that a given lamp is $\id_H$ tends to $0$ as $n$ tends to $\infty$, and the return probabilities of $R_n$ are again summable.

When $G=\Z$ and $S=\set{\pm 1}$ is the standard generating set, we are able to give precise asymptotic estimates on the range of the projection to $\Z$. This is done in \Sref{sec:Z-case}, where we prove \Tref{thm:Z-phase}, and in \Sref{sec:Z-case-local-time}, where we analyze the local time and prove \Tref{thm:Z-local-time}.

In \Sref{sec:homesick} we recall the definition of homesick random walks on rooted graphs, and prove two concentration results on their range, showing that they typically stay close to a ball of radius $C\log n$ (these bounds are not sharp, but sufficient for our applications). In \Sref{sec:general} we present the definition of stationary random walks on arbitrary lamplighter groups and establish their basic properties. In \Sref{sec:proofs} we combine the results of \Sref{sec:homesick} and \Sref{sec:general} to prove \Tref{thm:virt-Z} and \Tref{thm:no-transition}. Finally, in \Sref{sec:graphs} we discuss the generalization of stationary random walks to lamplighter graphs.

\section{\texorpdfstring{Recurrence and transience over $\Z$}{Recurrence and transience over Z}}\label{sec:Z-case}

In this section, we analyze the stationary random walk $R_n$ on $F\wr\Z$ for a given parameter $\frac{1}{2}<p<1$. As in the Introduction, we write $R_n=((L_n(i))_i,S_n)$, where $(L_n(i))_i\in\bigoplus_{\Z}F$ denotes the lamp configuration of $R_n$, and $S_n$ is a random walk on $\Z$ with drift $p$ towards the origin.

We denote by $\rho_0,\rho_1,\dots$ the successive times at which $S_n$ hits $0$. In other words, $\rho_0=0$, and for any $k\ge 1$ we write
\[
    \rho_k = \inf\set{n>\rho_{k-1}\suchthat S_n=0}.
\]
We remark that $\rho_1,\rho_2,\dots$ are almost surely finite, since the random walk $S_n$ is recurrent on $\Z$. We call each segment $R_{\rho_{k-1}},R_{\rho_{k-1}+1},\dots,R_{\rho_k}$ an \textbf{excursion}.

\begin{rem}\label{rem:rec-tran}
    $R_n$ is recurrent if and only if $\sum_{k=1}^{\infty}\Pr(R_{\rho_k}=\id)=\infty$. Indeed, for every $n$ with $R_n=\id$ we must have $S_n=0$, showing that the number of visits of the random walk $R_n$ to $\id$ is
    \[
        \sum_{n=1}^{\infty}\ind{\set{R_n=\id}} = \sum_{k=1}^{\infty}\ind{\set{R_{\rho_k}=\id}}.
    \]
\end{rem}

We therefore focus on studying the behaviour of the random walk $R_n$ at the sequence of times $(\rho_k)_k$.\medskip

For any $k\ge 1$, we write
\[
    M_k^+ = \max_{0\le j\le\rho_k} S_j
\]
and
\[
    M_k^- = \min_{0\le j\le \rho_k} S_j.
\]

\begin{prop}\label{prop:ret-prob-given-path}
  For any $k\ge 1$,
  \[
    \Pr(R_{\rho_k}=\id\mid S_0,S_1,\dots,S_{\rho_k}) = \left|F\right|^{-(M_k^+-M_k^-+1)}
  \]
\end{prop}

\begin{proof}
  We first recall that $S_{\rho_k}=0$, so $R_{\rho_k}=\id$ if and only if $L_{\rho_k}(i)=\id_F$ for all $i\in\Z$. By the definition of $R_n$, if $i\notin[M_k^-,M_k^+]$ then $L_n(i)=\id_F$.

  We claim that, conditional on $S_0, \ldots, S_{\rho_k}$, the variables $(L_{\rho_k}(i))_{i \in [M_k^-, M_k^+]}$ are independent and uniform on $F$, which immediately implies the proposition. Let us prove this (very intuitive, but subtle) claim.

  For each $i \in [M_k^-,M_k^+]$ we let $1 \leq n_i \leq \rho_k$ be the last time the random walk $S_n$ is at $i$, that is,
  \[
    n_i = \max \set{ 1 \leq t \leq \rho_k \suchthat  S_{t} = i } .
  \]
  We note that $0< n_i < \rho_k$ when $i \neq 0$, and by \eqref{eqn:L dfn} we have that $L_{\rho_k}(i) = U_{n_i+1}$ for all $0 \neq i \in [M_k^-,M_k^+]$. The lamp at $i=0$ is slightly special: here $n_0 = \rho_k$, and so $L_{\rho_k}(0) = V_{\rho_k}$.

  Therefore $R_{\rho_k} = \id$ if and only if $L_{\rho_k}(i) = \id_F$ for all $i \in [M_k^-,M_k^+]$, which is in turn is equivalent to
  \[
    \forall \ i \in [M_k^-,M_k^+] \setminus \{ 0\} \ , \ U_{n_i+1} = \id_F \qquad \textrm{and} \qquad
V_{\rho_k} = \id_F .
  \]
  Conditioned on $S_0, \ldots, S_{\rho_k}$, this last event has probability $\left|F\right|^{-|[M_k^-, M_k^+]| }$, proving the proposition.
\end{proof}

We next consider the distributions of $M_k^+$ and $M_k^-$. Since the random walks $(S_n)_n$ and $(-S_n)_n$ have the same distribution, also $M_k^+$ and $-M_k^-$ have the same distribution. We therefore focus on $M_k^+$. By definition, $M_k^+$ is the maximum of all positive excursions of $S_n$ until time $\rho_k$. For a single positive excursion, its maximum has the following distribution:

\begin{lem}\label{lem:max-ex}
  For any integer $x\ge 0$,
  \[
    \Pr(M_1^+\le x\mid S_1=1) = 1-\frac{\lambda-1}{\lambda^{x+1}-1},
  \]
  where $\lambda=\frac{p}{1-p}$.
\end{lem}

\begin{proof}
  Fix a non-negative integer $x\ge 0$. We write
  \[
    \tau_{x+1} = \inf\set{t\ge 0\suchthat S_t=x+1}
  \]
  and
  \[
    \tau_0^+ = \inf\set{t\ge 1\suchthat S_t=0}.
  \]
  Both $\tau_{x+1}$ and $\tau_0^+$ are almost surely finite, since the walk $(S_n)_n$ is recurrent. We next note that $(S_n)_{0\le n\le\tau_0^+\wedge\tau_{x+1}}$, conditioned on $S_1=1$, can be described using an electrical network, where the edge $(i,i+1)$ has resistance $\lambda^i$ for all $i\ge 0$. The vertices $0$ and $x+1$ are connected by a series of edges, and thus the effective resistance between them is
  \[
    \Reff(0\leftrightarrow x+1) = \sum_{i=0}^x\lambda^i = \frac{\lambda^{x+1}-1}{\lambda-1},
  \]
  and thus
  \[
    \Pr(M_1^+\le x\mid S_1=1) = \Pr(\tau_0^+ < \tau_{x+1}) = 1-\frac{1}{\Reff(0\leftrightarrow x+1)} = 1-\frac{\lambda-1}{\lambda^{x+1}-1}
  \]
  as required.
\end{proof}

Therefore, to study the distributions of $M_k^+$ and $M_k^-$, we need to know how many excursions are positive and how many are negative. We write, for any $k\ge 1$, the number of positive excursions until time $\rho_k$ as
\[
    N_k^+ = \left|\set{1\le j\le k\suchthat S_{\rho_{j-1}+1} > 0 } \right|
\]
and the number of negative excursions until time $\rho_k$ as
\[
    N_k^- = \left|\set{1\le j\le k\suchthat S_{\rho_{j-1}+1}< 0 } \right|,
\]
so that $N_k^++N_k^-=k$.
We remark for later that (by the strong Markov property) the excursions are independent, so the distribution of $N_k^+$ (as well as $N_k^-$) is binomial with parameters $k$ and $\frac{1}{2}$.

\begin{prop}\label{prop:ret-prob-given-nk}
  For any $k\ge 1$ and $0\le m\le k$,
  \begin{multline*}
    \Pr(R_{\rho_k}=\id\mid N_k^+=m) \\ = \frac{(\left|F\right|-1)^2}{\left|F\right|}\left(\sum_{a=1}^{\infty}\left|F\right|^{-a}\left(1-\frac{\lambda-1}{\lambda^a-1}\right)^m\right) \left(\sum_{b=1}^{\infty}\left|F\right|^{-b}\left(1-\frac{\lambda-1}{\lambda^b-1}\right)^{k-m}\right) ,
  \end{multline*}
 where $\lambda = \frac{p}{1-p}$.
\end{prop}

\begin{proof}
  We first note that, by the strong Markov property, the excursions $( (S_{\rho_{k-1}} , \ldots, S_{\rho_k}) )_{k \geq 1}$ are independent.  Therefore, conditioned on the value of $N_k^+$, the random variables $M_k^+$ and $M_k^-$ are independent. Thus, by \Pref{prop:ret-prob-given-path},
  \begin{align*}
    \Pr [ R_{\rho_k} = \id \mid N_k^+=m ] & = \E [ \left|F\right|^{-M_k^+ + M_k^- - 1} \mid N_k^+=m ] \\
    & = \frac{1}{\left|F\right|} \E [ \left|F\right|^{-M_k^+} \mid N_k^+=m ] \cdot \E[ \left|F\right|^{M_k^-} \mid N_k^+=m ].
  \end{align*}
  The measure preserving transformation $(S_n)_n \mapsto (-S_n)_n$ shows that the distribution of $M_k^-$ conditioned on $N_k^+=m$ is the same as that of $-M_k^+$ conditioned on $N_k^+=k-m$. Therefore,
  \[
    \Pr [ R_{\rho_k} = \id \mid N_k^+=m ] = \frac{1}{\left|F\right|} \E [ \left|F\right|^{-M_k^+} \mid N_k^+=m ] \cdot \E[ \left|F\right|^{-M_k^+} \mid N_k^+=k-m ] .
  \]
  We next observe that
  \begin{align*}
    \E [ \left|F\right|^{-M_k^+} \mid N_k^+=m ] & = \sum_{a=1}^\infty \left|F\right|^{-a} \Pr ( M_k^+=a \mid N_k^+=m ) \\
    & = \sum_{a=1}^\infty (\left|F\right|^{-a}-\left|F\right|^{-a-1}) \Pr ( M_k^+ \leq a \mid N_k^+=m ) \\
    & = (\left|F\right|-1)\sum_{a=1}^\infty \left|F\right|^{-a-1} \big( \Pr ( M_1^+ \leq a \mid S_1=1 ) \big)^m  \\
    & = (\left|F\right|-1)\sum_{a=2}^\infty \left|F\right|^{-a } \Big( 1 - \frac{\lambda-1}{\lambda^{a}- 1} \Big)^m \\
  \end{align*}
where in the third line we have used the fact that the excursions are independent, so that
$$ \Pr ( M_k^+ \leq a \mid N_k^+=m ) = \big( \Pr( M_1^+ \leq a \mid S_1=1 ) \big)^m , $$
and we have also used Lemma \ref{lem:max-ex}.
\end{proof}

We now analyze the factors in the above proposition.

\begin{lem}\label{lem:calculus}
  Set $\alpha=\frac{\log \left|F\right|}{\log \lambda}$ for $\lambda = \frac{p}{1-p}$. Then, for any integer $m$ larger than some constant depending on $p,\left|F\right|$, we have
  \[
    \sum_{a=1}^{\infty}\left|F\right|^{-a}\left(1-\frac{\lambda-1}{\lambda^a-1}\right)^m \simeq \frac{1}{m^{\alpha}},
  \]
  where the implied constants depend on $p$ and $\left|F\right|$.
\end{lem}

\begin{proof}
  For the upper bound, we note that
  \[
    \left|F\right|^{-a}\left(1-\frac{\lambda-1}{\lambda^a-1}\right)^m \le \left|F\right|^{-a}e^{-\frac{\lambda-1}{\lambda^a-1}m} \le \left|F\right|^{-a}e^{- \frac{\lambda-1}{\lambda^a}m}=e^{-a\log \left|F\right|-\frac{\lambda-1}{\lambda^a}m}.
  \]
  Let $f(a)=-a\log \left|F\right|-\frac{\lambda-1}{\lambda^a}m$. The maximum of $f$ is achieved when
  \[
    f'(a)=-\log \left|F\right|+\frac{(\lambda-1)m\log \lambda}{\lambda^a}=0
  \]
  i.e., when $a=a_0\coloneqq\log_\lambda\frac{(\lambda-1)m}{\alpha}$, so that $\lambda^{a_0}=\frac{(\lambda-1)m}{\alpha}$. In this case,
  \[
    e^{f(a_0)} = \left|F\right|^{-a_0}e^{-\frac{\lambda-1}{\lambda^{a_0}}m} = \frac{1}{(\lambda^{a_0})^{\alpha}}e^{-\frac{\lambda-1}{\lambda^{a_0}}m} = \left(\frac{\alpha}{(\lambda-1)em}\right)^{\alpha}.
  \]
  Note that if $a\le a_0$ then $\lambda^a\le \frac{(\lambda-1)m}{\alpha}$, and thus
  \[
    f(a) - f(a-1) = -\log \left|F\right| + \frac{(\lambda-1)^2}{\lambda^a}m \ge -\log \left|F\right| + \alpha(\lambda-1)>0
  \]
  since $\lambda-1>\log \lambda$. Therefore $e^{f(a-1)} \le C_1e^{f(a)}$ for $C_1=e^{\alpha(\lambda-1)-\log \left|F\right|}>1$, showing that
  \[
    \sum_{a\le a_0}e^{f(a)} \le \frac{C_1}{C_1-1}e^{f(a_0)} = \frac{C_2}{m^{\alpha}}
  \]
  for a constant $C_2$ that depends on $p,\left|F\right|$, while
  \[
    \sum_{a>a_0} e^{f(a)} = \sum_{a>a_0} \left|F\right|^{-a}e^{-\frac{\lambda-1}{\lambda^a}m} \le \sum_{a>a_0}\left|F\right|^{-a} = \frac{\left|F\right|^{-a_0}}{\left|F\right|-1} = \frac{1}{\left|F\right|-1}\left(\frac{\alpha}{\lambda-1}\right)^{\alpha}\frac{1}{m^{\alpha}}.
  \]
  Together we get
  \[
    \sum_{a=1}^{\infty}\left|F\right|^{-a}\left(1-\frac{\lambda-1}{\lambda^a-1}\right)^m \lesssim \frac{1}{m^{\alpha}}.
  \]

  For the lower bound, we have
  \begin{align*}
    \sum_{a=1}^{\infty}\left|F\right|^{-a}\left(1-\frac{\lambda-1}{\lambda^a-1}\right)^m & \ge \sum_{a=1}^{\infty}\left|F\right|^{-a}\left(1-\frac{1}{\lambda^{a-1}}\right)^m \\
    & \ge \left|F\right|^{-\ceil{a_0}+1}\left(1-\frac{1}{\lambda^{\ceil{a_0}}}\right)^m \\
    & \ge \left|F\right|^{-(a_0+2)}\left(1-\frac{1}{\lambda^{a_0}}\right)^m \\
    & = \frac{1}{\left|F\right|^2}\left(\frac{\alpha}{(\lambda-1)m}\right)^{\alpha}\left(1-\frac{\alpha \lambda}{(\lambda-1)m}\right)^m.
  \end{align*}
  Recall that $\left(1-\frac{x}{m}\right)^m\ge e^{-x}\left(1-\frac{x^2}{m}\right)$ for all $m\ge 1$ and $\left|x\right|\le m$, hence for $m\ge 2\left(\frac{\alpha \lambda}{\lambda-1}\right)^2$ we have $\left(1-\frac{\alpha \lambda}{(\lambda-1)m}\right)^m\ge \frac{1}{2}e^{-\frac{\alpha \lambda}{\lambda-1}}$, showing that
  \[
    \sum_{a=1}^{\infty}\left|F\right|^{-a}\left(1-\frac{\lambda-1}{\lambda^a-1}\right)^m \ge \frac{1}{2\left|F\right|^2}e^{-\frac{\alpha \lambda}{\lambda-1}}\frac{1}{m^{\alpha}},
  \]
  proving that
  \[
    \frac{1}{m^{\alpha}}\lesssim \sum_{a=1}^{\infty}\left|F\right|^{-a}\left(1-\frac{\lambda-1}{\lambda^a-1}\right)^m.\qedhere
  \]
\end{proof}

We can now conclude:

\begin{cor}\label{cor:ret-prob}
  Set $\alpha = \frac{\log \left|F\right|}{\log \lambda}$ for $\lambda = \frac{p}{1-p}$. Then, for any integer $k$ larger than some constant depending on $p,\left|F\right|$, we have
  \[
    \Pr(R_{\rho_k}=\id) \simeq \frac{1}{k^{2\alpha}},
  \]
  where the implied constants depend on $p,\left|F\right|$.
\end{cor}

\begin{proof}
  Note first that by Hoeffding's inequality,
  \[
    \Pr\left(\frac{k}{3}\le N_k^+\le\frac{2k}{3}\right) \le 2e^{-\frac{k}{18}}.
  \]
  For any $\frac{k}{3}\le m\le\frac{2k}{3}$, if $N_k^+=m$ then $\frac{k}{3}\le N_k^-=k-m\le\frac{2k}{3}$. Therefore, when $k$ is larger than a constant depending on $p,\left|F\right|$ we may apply \Lref{lem:calculus} and deduce
  \[
    \frac{c}{k^{\alpha}}\le \sum_{a=1}^{\infty}\left|F\right|^{-a}\left(1-\frac{\lambda-1}{\lambda^a-1}\right)^m \le \frac{C}{k^{\alpha}}
  \]
  and
  \[
    \frac{c}{k^{\alpha}}\le \sum_{b=1}^{\infty}\left|F\right|^{-b}\left(1-\frac{\lambda-1}{\lambda^b-1}\right)^{k-m} \le \frac{C}{k^{\alpha}}
  \]
  for $c,C>0$ constants which depend on $p,\left|F\right|$. Plugging this into \Pref{prop:ret-prob-given-nk}, we have
  \[
    \frac{c^2}{k^{2\alpha}} - 2e^{-\frac{k}{18}} \le \Pr(R_{\rho_k} = \id) \le \frac{C^2}{2k^{2\alpha}} + 2 e^{-\frac{k}{18}},
  \]
  which implies the corollary.
\end{proof}

We are ready to prove \Tref{thm:Z-phase}.

\begin{proof}[{Proof of \Tref{thm:Z-phase}}]
  By \Rref{rem:rec-tran}, the random walk $(R_n)_n$ is recurrent if and only if $\sum_{k=1}^{\infty}\Pr(R_{\rho_k}=\id)=\infty$. Applying \Cref{cor:ret-prob}, with $\alpha = \frac{\log\left|F\right|}{\log\lambda}$ and $\lambda = \frac{p}{1-p}$, this happens if and only if
  \[
    2 \alpha \leq 1 \iff \frac{p}{1-p}=\lambda \geq \left|F\right|^2 \iff p \geq \frac{\left|F\right|^2}{\left|F\right|^2+1}.\qedhere
  \]
\end{proof}

\section{\texorpdfstring{Local time over $\Z$}{Local time over Z}}\label{sec:Z-case-local-time}

Our next goal is to prove \Tref{thm:Z-local-time}. To do this, we first estimate the local time of $S_n$ at $0$.

\begin{prop}\label{prop:cramer}
  Let $m=\frac{2p}{2p-1}$. Then:
  \begin{enumerate}
    \item $\E[\rho_k]=km$ for every $k\ge 1$.
    \item For any $\eps>0$ there exists a constant $c=c(p,\eps)>0$ such that
    \[
        \Pr\left(\left|\rho_k-km\right|\ge\eps k\right)\le 2\exp(-ck)
    \]
    for all $k\ge 1$.
  \end{enumerate}
\end{prop}

\begin{proof}
  We first note that
  \begin{equation}\label{eq:rho_k-telescopic}
    \rho_k = \rho_1 + (\rho_2-\rho_1) + \cdots + (\rho_k-\rho_{k-1}),
  \end{equation}
 and by the strong Markov property, $\rho_1,\rho_2-\rho_1,\dots,\rho_k-\rho_{k-1}$ are independent and identically distributed. Therefore $\E[\rho_k]=\E[\rho_1]k$.

 Next, we note that the distribution of $\rho_1$ is
 \[
    \Pr(\rho_1=2t) = C_{t-1}p^t(1-p)^{t-1}
 \]
 for all $t\ge 1$, where $C_{t-1}=\frac{1}{t}\binom{2t-2}{t-1}$ is the $(t-1)$-th Catalan number. Its expected value is therefore
 \begin{align*}
   \E[\rho_1] & = \sum_{t=1}^{\infty}2tC_{t-1}p^t(1-p)^{t-1} = 2p\sum_{t=1}^{\infty}\binom{2t-2}{t-1}(p(1-p))^{t-1} \\
   & = 2p\frac{1}{\sqrt{1-4p(1-p)}} = \frac{2p}{2p-1} = m,
 \end{align*}
 where the third equality uses $\sum_{t=1}^{\infty}\binom{2t-2}{t-1}x^t = \frac{1}{\sqrt{1-4x}}$. With the above, we showed that $\E[\rho_k]=\E[\rho_1]k=km$, proving the first part of the proposition.

 For the second part, we note that the moment generating function of $\rho_1$ is
 \begin{align*}
   \E[e^{s\rho_1}] & = \sum_{t=1}^{\infty}e^{2st}C_{t-1}p^t(1-p)^{t-1} = pe^{2s}\sum_{t=1}^{\infty}C_{t-1}(p(1-p)e^{2s})^{t-1} \\
   & = pe^{2s}\frac{1-\sqrt{1-4p(1-p)e^{2s}}}{2p(1-p)e^{2s}} = \frac{1-\sqrt{1-4p(1-p)e^{2s}}}{2(1-p)}.
 \end{align*}
 Write $\psi(s)\coloneqq\log\E[e^{s\rho_1}]-sm$. The function $\psi$ is differentiable continuously twice whenever $s<\frac{1}{2}\log\frac{1}{4p(1-p)}\eqqcolon s_0$. We also note that $\psi(0)=\psi'(0)=0$, and write $c_0=\frac{1}{2}\sup_{\left|s\right|\le \frac{1}{2}s_0}\left|\psi''(s)\right|$. By Taylor's expansion, for any $\left|s\right|\le\frac{1}{2}s_0$ we have
 \[
    \psi(s) = \psi(0) + \psi'(0)s + \frac{1}{2}\psi''(s')s^2 = \frac{1}{2}\psi''(s')s^2 \le c_0s^2
 \]
 where $s'$ is a point which lies between $0$ and $s$.

 We may now apply Markov's inequality and \eqref{eq:rho_k-telescopic} to deduce that for every $\eps>0$,
 \[
    \Pr(\rho_k\ge (m+\eps)k) \le \frac{\E[e^{s\rho_k}]}{e^{sk(m+\eps)}} = \frac{\E[e^{s\rho_1}]^k}{e^{sk(m+\eps)}} = e^{k(\psi(s)-\eps s)} \le e^{k(c_0s^2-\eps s)}
 \]
 and
 \[
    \Pr(\rho_k\le (m-\eps)k) \le \frac{\E[e^{-s\rho_k}]}{e^{-sk(m-\eps)}} = \frac{\E[e^{-s\rho_1}]^k}{e^{-sk(m-\eps)}} = e^{k(\psi(-s)-\eps s)} \le e^{k(c_0s^2-\eps s)}.
 \]
 Taking $s=\min\set{\frac{\eps}{2c_0},\frac{1}{2}s_0}$, we get
 \[
    \Pr(\left|\rho_k-km\right|\ge\eps k) \le 2e^{-ck}
 \]
 for some constant $c>0$ depending on $p,\eps$, as required.
\end{proof}

\begin{proof}[{Proof of \Tref{thm:Z-local-time}}]
Let $m = \frac{2p}{2p-1}$.
  Writing $A=\set{\rho_{\floor{n/2m}}\le n\le \rho_{\ceil{2n/m}}}$, we have
  \begin{align*}
    \E[\xi(\id,n)] & = \E[\xi(\id,n)\ind{A}] + \E[\xi(\id,n)\ind{A^c}] \\
    & \le \E[\xi(\id,\rho_{\ceil{2n/m}})] + n\Pr(A^c)
  \end{align*}
  and
  \begin{align*}
    \E[\xi(\id,n)] & \ge \E[\xi(\id,n)\ind{A}] \ge \E[\xi(\id,\rho_{\floor{n/2m}})\ind{A}]\\
    & = \E[\xi(\id,\rho_{\floor{n/2m}})] - \E[\xi(\id,\rho_{\floor{n/2m}})\ind{A^c}]\\
    & \ge \E[\xi(\id,\rho_{\floor{n/2m}})] - \frac{n}{2m}\Pr(A^c).
  \end{align*}
  By \Pref{prop:cramer}, there is a constant $c>0$, depending only on $p$, such that $\Pr(A^c) \le 2e^{-cn}$, and thus
  \[
    \E[\xi(\id,\rho_{\floor{n/2m}})] -\frac{1}{m}ne^{-cn} \le \E[\xi(\id,n)] \le \E[\xi(\id,\rho_{\ceil{2n/m}})] + 2ne^{-cn}.
  \]
  Finally, \Cref{cor:ret-prob} shows that for $\alpha = \frac{\log 2}{\log p - \log (1-p)}$,
  \[
    \E[\xi(\id,\rho_k)] = 1 + \sum_{j=1}^k\Pr(R_{\rho_j}=\id) \simeq \sum_{j=1}^k \frac{1}{j^{2\alpha}}\simeq \begin{cases} 1, & p<\frac{\left|F\right|^2}{\left|F\right|^2+1} \\ \log n, & p=\frac{\left|F\right|^2}{\left|F\right|^2+1} \\ n^{1-2\alpha}, & p>\frac{\left|F\right|^2}{\left|F\right|^2+1} \end{cases}
  \]
  for both $k\in\set{\floor{n/2m},\ceil{2n/m}}$, and the theorem follows.
\end{proof}

\section{Homesick random walks}\label{sec:homesick}

We turn to study the stationary random walks over an arbitrary finitely generated group~$G$. Before we do that, we recall the notion of homesick random walks on graphs. These random walks were first studied on trees (see e.g.\ \cite{BerrettiSokal85,Krug88,LawlerSokal88,Lyons90}), and were later considered on general Cayley graphs \cite{Lyons95,LyonsPemantlePeres96,Revelle01}.

Let $\Gamma=(V,E)$ be a locally finite graph with root $o\in V$, and let $\lambda\ge 1$ be a real number. The \textbf{$\lambda$-homesick random walk} $(Z_n)_{n=0}^{\infty}$ on $\Gamma$ is then defined as follows. We first set $Z_0=o$. For each $v\in V$, let $v_1,\dots,v_j$ denote the neighbors of~$v$ with $\dist(o,v_i)=\dist(o,v)-1$ for all $1\le i\le j$ (notice that $j\ge 1$ unless $v=o$), and let $v_1',\dots,v_k'$ denote the other neighbors of $v$. We then define the transition probability from $v\in V$ to $w\in V$ by
\[
    T(v,w) = \begin{cases} \frac{\lambda}{\lambda j + k}, & w\in\set{v_1,\dots,v_j} \\ \frac{1}{\lambda j + k}, & w\in\set{v_1',\dots,v_k'} \\ 0, & \text{otherwise}.  \end{cases}
\]
When $\lambda=1$, this is a simple random walk on $\Gamma$, and otherwise has a drift towards the root $o$. Alternatively, the $\lambda$-homesick random walk can be described by an electrical network, where the resistance of an edge $(v,w)\in E$ is $\lambda^{\min\set{\dist(o,v),\dist(o,w)}}$.

We denote by
\[
    B_r = \set{ v \in V \suchthat \dist(o,v) \le r }
\]
and
\[
    K_r = \set{ v \in V \suchthat \dist(o,v) = r }
\]
the ball and sphere of radius~$r$ around the root~$o$, respectively.
Note that $B_r$ is well defined also for non-integer~$r$, in which case $B_r = B_{\floor{r}}$.
For a nonempty subset $A \sub V$, we write
\[
    \partial_E A = \set{ (u,v) \in E \suchthat u \in A, \, v \notin A }
\]
for its edge boundary.

We define the stopping times
\[
    \tau_o^+ = \inf \set{ n \ge 1 \suchthat Z_n = o }
\]
and, for any nonempty subset $A \sub V$,
\[
    \tau_A = \inf \set{ n \ge 0 \suchthat Z_n \in A }.
\]
Both may be infinite if the walk $(Z_n)$ is transient.
For a vertex $v \in V$, we also write $\tau_v = \tau_{\set{v}}$.

We now give two concentration results on the size of $\range(Z_n)=\set{Z_0,\dots,Z_n}$. While $Z_n$ is supported on the whole ball of radius $n$, when the drift is large enough we expect it to remain, with high probability, in a ball of radius $C\log n$ (where $C$ depends on $\lambda$). This is formulated in the following proposition:

\begin{prop}\label{prop:homesick-range-up}
  Let $\Gamma=(V,E)$ be a locally finite graph with root $o$, let $\lambda\ge 1$, and let $Z_n$ be the $\lambda$-homesick random walk on $(\Gamma,o)$. Then
  \[
    \Pr\left(\max_{i<\tau_o^+}\dist(o,Z_i)\ge r\right) \le \frac{1}{\deg o}\left(\sum_{i=0}^{r-1}\left|\partial_E B_i\right|^{-1}\lambda^i\right)^{-1}
  \]
  for every $r\ge 1$.
\end{prop}

\begin{proof}
  We first recall from the theory of electrical networks that
  \[
    \Pr\left(\max_{i<\tau_o^+}\dist(o,Z_i)\ge r\right) = \Pr(\tau_{K_r}<\tau_o^+) = \frac{1}{\deg o\;\Reff(o\leftrightarrow z_r;\Gamma_r)}
  \]
  where $\Gamma_r$ denotes the graph obtained from $\Gamma$ by gluing all vertices of $K_r$ in $\Gamma$ to one vertex $z_r$. Let $\Gamma'$ denote the graph obtained from $\Gamma$ by gluing each sphere $K_d$ in $\Gamma$ to a vertex $z_d$. Since $\Gamma'$ is obtained from $\Gamma_r$ by gluing vertices, we have
  \[
    \Reff(o\leftrightarrow z_r;\Gamma') \le \Reff(o\leftrightarrow z_r;\Gamma_r),
  \]
  and so
  \[
    \Pr\left(\max_{i<\tau_o^+}\dist(o,Z_i)\ge r\right) \le \frac{1}{\deg o\;\Reff(o\leftrightarrow z_r;\Gamma')}.
  \]
  Now $\Gamma'$ is a one-sided line, where the number of edges from $z_i$ to $z_{i+1}$ is $\left|\partial_E B_i\right|$, and the resistance of each one is $\lambda^i$. Since these are all connected in parallel, we have
  \[
    \Reff(z_i\leftrightarrow z_{i+1};\Gamma') = \left|\partial_E B_i\right|^{-1}\lambda^i,
  \]
  and since $z_0=o,z_1,\dots,z_r$ are connected in a series we have
  \[
    \Reff(o\leftrightarrow z_r;\Gamma') = \sum_{i=0}^{r-1}\left|\partial_E B_i\right|^{-1}\lambda^i.
  \]
  We therefore conclude that
  \[
    \Pr\left(\max_{i<\tau_o^+}\dist(o,Z_i)\ge r\right) \le \frac{1}{\deg o}\left(\sum_{i=0}^{r-1}\left|\partial_E B_i\right|^{-1}\lambda^i\right)^{-1},
  \]
  completing the proof.
\end{proof}

We next show that, with high probability, $Z_n$ covers a positive proportion of a ball of radius $c\log n$. If $Z_n$ is transient, then its range will be linear in $n$ with high probability, so we focus on the case where $Z_n$ is recurrent. Similarly to the previous sections, we write $\rho_0,\rho_1,\dots$ for the successive times at which $Z_n$ hits $o$, i.e., $\rho_0=0$ and for every $k\ge 1$ we have
\[
    \rho_k = \inf\set{n>\rho_{k-1}\suchthat Z_n=o}.
\]

\begin{prop}\label{prop:homesick-range-down}
  Let $\Gamma=(V,E)$ be a graph with root $o$, let $\lambda>1$, and let $Z_n$ be the $\lambda$-homesick random walk on $(\Gamma,o)$. Assume that $Z_n$ is recurrent. Then, for every $k\ge \left(\frac{5(\lambda-1)}{\deg o}\right)^{1/c}$ and $0<c<1$,
  \[
    \Pr\left(\left|\range(Z_{\rho_k})\right| \le \frac{N}{4}\right) \le e^{-\frac{N}{8}}
  \]
  where
  \[
    N = \min\set{\left|B_{\frac{c\log k}{\log\lambda}}\right|,\floor{\frac{\lambda-1}{\deg o}k^{1-c}}}.
  \]
\end{prop}

\begin{proof}
  We first note that for any $v\in V$, there is a path $v_0=o,v_1,\dots,v_d=v$ for $d=\dist(o,v)$ with resistance $\sum_{i=0}^{d-1}\lambda^i\le\frac{\lambda^d}{\lambda-1}$. Therefore
  \[
    \Reff(o\leftrightarrow v) \le  \frac{1}{\lambda-1}\lambda^{\dist(o,v)}
  \]
  for all $v\in V$, showing that
  \[
    \Pr(\tau_v<\tau_o^+) = \frac{1}{\deg o\;\Reff(o\leftrightarrow v)} \ge \frac{\lambda-1}{\deg o}\lambda^{-\dist(o,v)}
  \]
  for all $v\in V$.

  We write $r=\frac{c\log k}{\log\lambda}$, $N=\min\set{\left|B_r\right|,\floor{\frac{\deg o}{\lambda-1}k^{1-c}}}$, and $k'=\floor{\frac{k}{N}}$. We take~$N$ distinct vertices $v_1,\dots,v_N\in B_r$, and for each $1\le j\le N$ we denote by $A_j$ the event that $Z_n$ reaches $v_j$ between times $\rho_{(j-1)k'}$ and $\rho_{jk'}$. The events $A_1,\dots,A_N$ are independent, since they depend on mutually disjoint intervals of time. Also, by the above, at each of the $k'$ excursions between $\rho_{(j-1)k'}$ and $\rho_{jk'}$ there is a probability of at least $\frac{\lambda-1}{\deg o}\lambda^{-r}=\frac{\lambda-1}{\deg o}k^{-c}$ that $R_n$ reaches $v_j$ during that excursion. Therefore
  \[
    \Pr(A_j) \ge 1 - \left(1-\frac{\lambda-1}{\deg o}k^{-c}\right)^{k'} \ge 1 - \frac{5}{4}\left(1-\frac{\lambda-1}{\deg o}k^{-c}\right)^{k/N} \ge 1 - \frac{5}{4}e^{-1} > \frac{1}{2}
  \]
  where the second inequality follows from $k\ge \left(\frac{5(\lambda-1)}{\deg o}\right)^{1/c}$ and $k'\ge\frac{k}{N}-1$, and the third inequality follows from $N\le \frac{\deg o}{\lambda-1}k^{1-c}$.

  Finally, we note that $\left|\range(Z_{\rho_k})\right| \ge \sum_{j=1}^N\ind{A_j}$, so $\left|\range(Z_{\rho_k})\right|$ stochastically dominates a binomial random variable with parameters $N$ and $\frac{1}{2}$. By Chernoff's inequality,
  \[
    \Pr\left(\left|\range(Z_{\rho_k})\right| \le \frac{N}{4}\right) \le e^{-\frac{N}{8}}
  \]
  completing the proof.
\end{proof}

\section{Stationary random walks: the general case}\label{sec:general}

We are now ready to define the stationary random walks in the general case. Let~$G$ be a finitely generated group with a finite symmetric generating set $S$, and let $H$ be a nontrivial finitely generated group with a finitely supported, symmetric and non-degenerate probability measure $\mu$. For any $\lambda\ge 1$, we define the \textbf{$\lambda$-stationary random walk} on $H\wr G$ as $R_n=((L_n(g))_{g\in G},Z_n)$, where $Z_n$ is the $\lambda$-homesick random walk on the Cayley graph $\Cay(G,S)$, and $L_{n+1}$ is obtained from $L_n$ via%
\begin{equation}
\label{eqn:L dfn2}
L_{n+1}(g) =
\begin{cases}
L_n(g) & \textrm{ if } g \not\in \{ Z_n , Z_{n+1} \} , \\
L_n(g)U_{n+1} & \textrm{ if } g=Z_n , \\
L_n(g)V_{n+1} & \textrm{ if } g=Z_{n+1},
\end{cases}
\end{equation}
where $(U_n)_n,(V_n)_n$ are $\mu$-i.i.d.\ random variables. We note that if $H$ is finite and $\mu$ is uniform on $H$, then this has the same distribution as the one defined via \eqref{eqn:L dfn}.

Since we are interested in studying the recurrence of $R_n$, we will often assume that $Z_n$ is recurrent. In this case we write $\rho_0=0,\rho_1,\dots$ for the successive times at which $Z_n$ hits $\id_G$.

\begin{rem}\label{rem:rec-tran-gen}
  As in the case of $F\wr\Z$, the random walk $R_n$ is recurrent if and only if $Z_n$ is recurrent and $\sum_{k=1}^{\infty}\Pr(R_{\rho_k}=\id)=\infty$.
\end{rem}

We next rephrase \Pref{prop:ret-prob-given-path} for the general case. To do this, we write the local time of $Z_n$ at an element $g\in G$ until time $n$ as
\[
  \xi_G(g,n) = \left|\set{0\le t\le n\suchthat Z_t=g}\right|.
\]
We further write $\mu^{*n}$ for the $n$-fold convolution power of $\mu$, i.e.,
\[
    \mu^{*n}(h) = \sum_{h_1\cdots h_n=h}\mu(h_1)\cdots\mu(h_n)
\]
for every $h\in H$ and $n\ge 1$.

\begin{prop}\label{prop:ret-prob-gen}
  Assume that $Z_n$ is recurrent. Then, for any $k\ge 1$,
  \[
    \Pr(R_{\rho_k}=\id \suchthat Z_0,\dots,Z_{\rho_k}) = \prod_{g\in \range(Z_{\rho_k})}\mu^{*(2\xi_G(g,\rho_k))}(\id_H).
  \]
\end{prop}

\begin{proof}
  This is essentially the same proof as the proof of \Pref{prop:ret-prob-given-path}. For any $g\in G$ we write
  \[
    \Xi(g,\rho_k) = \set{0\le t< \rho_k\suchthat Z_t=g}.
  \]
  By the definition of $R_n$, if $g\notin\range(Z_{\rho_k})$ then $L_{\rho_k}(g)=\id_H$. If $\id_G\ne g\in \range(Z_{\rho_k})$, then
  \[
    L_{\rho_k}(g) = \prod_{t\in\Xi(g,\rho_k)}V_tU_{t+1},
  \]
  whereas for $g=\id_G$ we have
  \[
    L_{\rho_k}(\id_G) = U_1\left(\prod_{i=1}^{k-1} V_{\rho_i}U_{\rho_i+1}\right)V_{\rho_k}.
  \]
  It follows that, conditional on $Z_0,\dots,Z_{\rho_k}$, the variables $(L_{\rho_k}(g))_{g\in\range(Z_{\rho_k})}$ are independent. Furthermore, the distribution of $L_{\rho_k}(g)$ is the product of $2\xi_G(g,\rho_k)$ independent $\mu$-random variables, and thus $\Pr(L_{\rho_k}(g)=\id_H) = \mu^{*(2\xi_G(g,\rho_k))}(\id_H)$. We therefore have
  \begin{align*}
    \Pr(R_{\rho_k}=\id \suchthat Z_0,\dots,Z_{\rho_k}) &= \Pr(\forall g\in G:L_{\rho_k}(g)=\id_H) \\
    & = \prod_{g\in \range(Z_{\rho_k})}\mu^{*(2\xi_G(g,\rho_k))}(\id_H),
  \end{align*}
  completing the proof.
\end{proof}

Before studying stationary random walks further, we describe them via an appropriate electrical network.

\begin{obs}
  The $\lambda$-stationary random walk $R_n$ on $H\wr G$ can also be described using an electrical network, where the underlying graph $\Gamma$ of the network is the Cayley graph of $H\wr G$ with respect to the switch-walk-switch generators defining~$R_n$. In details, two elements $a=((L(g))_g,x)\in H\wr G$ and $a'=((L'(g))_g,x')\in H\wr G$ are connected by an edge in $\Gamma$ if and only if:
  \begin{itemize}
    \item $x'=xs$ for some $s\in S$;
    \item $L(g)=L'(g)$ for all $g\in G\setminus\set{x,x'}$;
    \item and $\mu(L(g)^{-1}L'(g))>0$ for both $g\in\set{x,x'}$.
  \end{itemize}
  The resistance of the edge $e=(a,a')$ is then achieved by multiplying the resistances of the $\mu$-random walk twice (one time for each switch) and the resistance from $Z_n$, i.e.,
  \[
    r_e=\lambda^{\min\{|x|,|x'|\}}\mu(L(x)^{-1}L'(x))^{-1}\mu(L(x')^{-1}L'(x'))^{-1}.
  \]
\end{obs}

We apply this viewpoint to prove two claims on the recurrence of the stationary random walks. First, we show that the recurrence of $R_n$ does not depend on the choice of the measure $\mu$, and in fact only depends on $H$ via its order $\left|H\right|$.

\begin{prop}\label{prop:independent}
  Let $G$ be an infinite finitely generated group with a finite symmetric generating set $S$, and let $\lambda\ge 1$ be a real number. Let $H,H'$ be two finite groups of the same order, and let $\mu,\mu'$ be symmetric and non-degenerate probability measures on $H,H'$ respectively. Write $R_n,R_n'$ for the $\lambda$-stationary random walks on $H\wr G,H'\wr G$ with respect to $\mu,\mu'$. Then $R_n$ is recurrent if and only if $R_n'$ is recurrent.
\end{prop}

\begin{proof}
  By symmetry, it suffices to show that recurrence of $R_n$ implies recurrence of $R_n'$. Denote by $\Gamma=(H\wr G,E,(r_e)_e)$ and $\Gamma'=(H'\wr G,E',(r'_e)_e)$ the electrical networks corresponding to $(R_n)_n$ and $(R_n')_n$. For a vertex $a=((L(g))_g,x)\in H\wr G$, the total conductance of $\Gamma$ at $a$ satisfies
  \[
    \alpha^2\lambda^{-|x|-1} \le c_a=\sum_{(a,b)\in E} r_{(a,b)}^{-1} \le \lambda^{-|x|+1},
  \]
  where $\alpha=\min_{h\in\supp(\mu)}\mu(h)$. Similarly, for $a'=((L'(g))_g,x)\in H'\wr G$ we have
  \[
    \beta^2\lambda^{-|x|-1} \le c'_{a'}=\sum_{(a',b')\in E'} (r'_{(a',b')})^{-1} \le \lambda^{-|x|+1},
  \]
  where $\beta=\min_{h\in\supp(\mu')}\mu'(h)$.

  Fix a bijection $f\colon H\to H'$, which extends to a bijection $\tilde f\colon H\wr G\to H'\wr G$ by applying $f$ on each lamp. From the above bounds it follows that
  \[
    \frac{\alpha^2}{\lambda^2} \le \frac{c_a}{c'_{\tilde{f}(a)}} \le \frac{\lambda^2}{\beta^2}
  \]
  for every $a\in H\wr G$.

  Choose an odd integer $k\ge 1$ with $\supp(\mu^{*k})=H$ (this is possible since $\mu$ is symmetric and non-degenerate), and let $\Gamma^{(k)}=(H\wr G,E^{(k)},(r^{(k)}_e)_e)$ be the network corresponding to $(R_{kn})_n$. We claim that there exists a positive number $C>0$, depending only on $|S|,\lambda,\mu,\mu'$, such that
  \begin{equation}\label{eq:res-ineq}
    r^{(k)}_{(a_1,a_2)} \le C\, r'_{(\tilde{f}(a_1),\tilde{f}(a_2))}
  \end{equation}
  for all $a_1,a_2\in H\wr G$.

  To see this, write $a_i=((L_i(g))_g,x_i)$ for $i=1,2$. If $(\tilde{f}(a_1),\tilde{f}(a_2))\notin E'$, then~ \eqref{eq:res-ineq} is trivial, so we assume $(\tilde{f}(a_1),\tilde{f}(a_2))\in E'$. We note that
  \[
    \Pr(R'_1=\tilde{f}(a_2)\suchthat R'_0=\tilde{f}(a_1)) \le \Pr(Z_1=x_2\suchthat Z_0=x_1) \le \frac{\lambda}{\lambda+|S|-1},
  \]
  while
  \begin{align*}
    \Pr(R_k=a_2\suchthat R_0=a_1) &= \mu^{*k}(L_1(g_1)^{-1}L_2(g_1)) \,\mu^{*k}(L_1(g_2)^{-1}L_2(g_2)) \,\Pr(Z_k=x_2\suchthat Z_0=x_1) \\
    &\ge \gamma^2 \,\Pr(Z_1=x_2\suchthat Z_0=x_1)^{(k+1)/2}\,\Pr(Z_1=x_2\suchthat Z_0=x_1)^{(k-1)/2} \\
    &\ge \frac{\gamma^2}{(\lambda+|S|-1)^k}
  \end{align*}
  for $\gamma=\min_{h\in H}\mu^{*k}(h)>0$. Therefore
  \begin{align*}
    r^{(k)}_{(a_1,a_2)} &= \frac{1}{c_{a_1}\Pr(R_k=a_2\suchthat R_0=a_1)}
    \le \frac{(\lambda+|S|-1)^k}{\gamma^2 c_{a_1}} \\
    &\le \frac{\lambda^2(\lambda+|S|-1)^k}{\alpha^2\gamma^2 c'_{\tilde{f}(a_1)}}
    \le \frac{\lambda^3(\lambda+|S|-1)^{k-1}}{\alpha^2\gamma^2 c'_{\tilde{f}(a_1)}\,\Pr(R'_1=\tilde{f}(a_2)\suchthat R'_0=\tilde{f}(a_1))} \\
    &= \frac{\lambda^3(\lambda+|S|-1)^{k-1}}{\alpha^2\gamma^2}\,r'_{(\tilde{f}(a_1),\tilde{f}(a_2))},
  \end{align*}
  establishing \eqref{eq:res-ineq}.

  By Rayleigh's monotonicity law,
  \[
    \Reff(\id\leftrightarrow\infty;\Gamma^{(k)}) \le C\,\Reff(\tilde{f}(\id)\leftrightarrow\infty;\Gamma').
  \]
  Since $R_n$ is recurrent, so is $R_{kn}$, and therefore
  \[
    \Reff(\id\leftrightarrow\infty;\Gamma^{(k)})=\infty,
  \]
  which forces
  \[
    \Reff(\tilde{f}(\id)\leftrightarrow\infty;\Gamma')=\infty.
  \]
  Thus $R_n'$ is recurrent.
\end{proof}

Next, we show that recurrent is a monotone property in $\lambda$:

\begin{prop}\label{prop:monotonicity}
  Let $G$ be a finitely generated group, let $H$ be a nontrivial finite group, and let $\lambda>1$ be a real number. If the $\lambda$-stationary random walk on $H\wr G$ is recurrent, then so is the $\lambda'$-stationary random walk for every $\lambda'>\lambda$.
\end{prop}

\begin{proof}
  Let $\Gamma$ and $\Gamma'$ denote the electrical networks representing the $\lambda$- and $\lambda'$-stationary random walks on $H\wr G$ respectively. These networks have the same underlying graph, and the resistances on the edges satisfy $r_e\le r'_e$ for all $e\in E$. The assumption that the $\lambda$-stationary random walk on $H\wr G$ is recurrent implies
  \[
    \Reff(\id\leftrightarrow\infty;\Gamma) = \infty.
  \]
  By Rayleigh's monotonicity law, it follows that
  \[
    \Reff(\id\leftrightarrow\infty;\Gamma) \le \Reff(\id\leftrightarrow\infty;\Gamma'),
  \]
  so
  \[
    \Reff(\id\leftrightarrow\infty;\Gamma') = \infty
  \]
  as well, which in turn implies that the $\lambda'$-stationary random walk on $H\wr G$ is recurrent.
\end{proof}

\section{Proofs of \Tref{thm:virt-Z} and \Tref{thm:no-transition}}\label{sec:proofs}

We are now ready to prove \Tref{thm:virt-Z} and \Tref{thm:no-transition}. Throughout the section, we fix an infinite finitely generated group $G$ with a finite symmetric generating set $S$. We also fix a nontrivial and finitely generated group $H$, and a finitely supported, symmetric and non-degenerate probability measure $\mu$ on $H$. For a given $\lambda\ge 1$, we write $R_n$ for the $\lambda$-stationary random walk on $H\wr G$. Throughout this section, we write $B_r$ for the ball of radius $r$ around $\id_G$ in the Cayley graph of $G$.

We begin with \Tref{thm:virt-Z}, stating that when $G$ is virtually-$\Z$ and $H$ is finite, there is a phase transition in the recurrence of the stationary random walks.

\begin{proof}[Proof of \Tref{thm:virt-Z}]
  Assume that $G$ is virtually-$\Z$ and that $H$ is finite. By \Pref{prop:independent}, we may assume that $\mu$ is the uniform measure on $H$. We will prove that there exist $1<\lambda_1<\lambda_2<\infty$ such that $R_n$ is transient if $\lambda<\lambda_1$, and is recurrent if $\lambda\ge\lambda_2$. We note that by the result of Lyons \cite{Lyons95}, $Z_n$ is recurrent since $G$ has linear growth.

  We first prove the existence of $\lambda_1$. Since $G$ is infinite, $\left|B_r\right|\ge r$ for every $r\ge 1$. Therefore, by \Pref{prop:homesick-range-down} with $c=\frac{1}{2}$, for sufficiently large $k$ we have
  \[
    \Pr\left(\left|\range(Z_{\rho_k})\right| \le \frac{1}{4}N_k\right) \le e^{-\frac{N_k}{8}}
  \]
  for
  \[
    N_k = \min\set{\left|B_{\frac{\log k}{2\log\lambda}}\right|,\floor{\frac{\lambda-1}{\left|S\right|}\sqrt{k}}} \ge \min\set{\frac{\log k}{2\log\lambda}, \floor{\frac{\lambda-1}{\left|S\right|}\sqrt{k}}}.
  \]
  When $k$ is sufficiently large we have $N_k \ge \frac{\log k}{2\log\lambda}$, so \Pref{prop:ret-prob-gen} implies
  \begin{align*}
    \Pr(R_{\rho_k}=\id) & \le \Pr\left(R_{\rho_k}=\id\suchthat \left|\range(Z_{\rho_k})\right|\ge \frac{N_k}{4}\right) + \Pr\left(\left|\range(Z_{\rho_k})\right|\le \frac{N_k}{4}\right) \\
    & \le \left|H\right|^{-N_k/4} + e^{-N_k/8} \le \left|H\right|^{-\log k/8\log\lambda} + e^{-\log k/16\log\lambda} \\
    & = k^{-\log\left|H\right|/8\log\lambda} + k^{-1/16\log\lambda} \le 2k^{-\log\left|H\right|/8\log\lambda}
  \end{align*}
  for sufficiently large $k$. Therefore, if $\lambda<\lambda_1\coloneqq\left|H\right|^{1/8}$, then $\sum_{k=1}^{\infty}\Pr(R_{\rho_k}=\id)<\infty$, and thus $R_n$ is transient for all $\lambda<\lambda_1$.

  We turn to prove the existence of $\lambda_2$. We note that since $G$ is virtually-$\Z$, it follows that $G$ and $\Z$ are quasi-isometric. In particular, there exists $M>0$ such that $G$ has at most $M$ edges of any given distance from $\id_G$, i.e., $\left|\partial_E B_r\right|\le M$ for all $r\ge 1$. Furthermore, there exists a constant $C>0$ such that $\left|B_r\right|\le Cr$ for all $r\ge 1$.

  We will show that $\range(Z_{\rho_k})$ is contained in a small ball with high probability. To do this, we first apply \Pref{prop:homesick-range-up} and deduce that
  \[
    \Pr\left(\max_{i\le\rho_1}\left|Z_i\right|\ge r\right) \le \frac{1}{\left|S\right|}\left(\sum_{i=0}^{r-1}M^{-1}\lambda^i\right)^{-1} \le \frac{M}{\left|S\right|\lambda^{r-1}}.
  \]
  Fix $k\ge 2$. Since the different excursions are independent, it follows that
  \[
    \Pr\left(\max_{i\le\rho_k}\left|Z_i\right|<r\right) \ge \left(1 - \frac{M}{\left|S\right|\lambda^{r-1}}\right)^k
  \]
  for every $r\ge 1$. In particular, taking $r=r_k\coloneqq\frac{\log (kM/\left|S\right|)}{\log\lambda}+1$, we have
  \[
    \Pr\left(\max_{i\le\rho_k}\left|Z_i\right|<r_k\right) \ge \left(1 - \frac{1}{k}\right)^k \ge \frac{1}{4}.
  \]
  We note that if $\max_{i\le\rho_k}\left|Z_i\right|< r_k$, then $\range(Z_{\rho_k})\sub B_{r_k-1}$, and thus $\left|\range(Z_{\rho_k})\right| \le \left|B_{r_k-1}\right|\le C(r_k-1)$. Therefore
  \[
    \Pr\left(\left|\range(Z_{\rho_k})\right| \le C(r_k-1) \right)\ge \frac{1}{4},
  \]
  so we apply \Pref{prop:ret-prob-gen} again and deduce
  \begin{align*}
    \Pr(R_{\rho_k}=\id) & \ge \Pr\left(R_{\rho_k}=\id\suchthat \left|\range(Z_{\rho_k})\right|\le C(r_k-1)\right)\Pr\left(\left|\range(Z_{\rho_k})\right|\le C(r_k-1)\right) \\
    & \ge \frac{1}{4}\left|H\right|^{-C(r_k-1)} = \frac{1}{4}\left|H\right|^{-\frac{C\log k}{\log\lambda}+C'} = \frac{1}{4}\left|H\right|^{C'}k^{-\frac{C\log\left|H\right|}{\log\lambda}}.
  \end{align*}
  It follows that if $\lambda\ge\lambda_2\coloneqq\left|H\right|^C$ then $\sum_{k=1}^{\infty}\Pr(R_{\rho_k}=\id)=\infty$, and thus $R_n$ is recurrent for all $\lambda\ge\lambda_2$.

  We have proven the existence of $1<\lambda_1<\lambda_2<\infty$ such that the $\lambda$-stationary random walk $R_n$ is transient if $\lambda<\lambda_1$, and is recurrent if $\lambda\ge\lambda_2$, and thus the theorem follows by \Pref{prop:monotonicity}.
\end{proof}

We now turn to prove \Tref{thm:no-transition}, i.e., that there is no phase transition if $G$ is not virtually-$\Z$ or if $H$ is infinite. We note that for $\lambda=1$, the $\lambda$-stationary random walk on $H\wr G$ is just a simple random walk on this group, which is known to be transient. We therefore focus on the case $\lambda>1$, and prove the theorem under each assumption separately.

\begin{proof}[Proof of \Tref{thm:no-transition}, $G$ non-virtually-$\Z$ case]
  Suppose that $G$ is not virtually-$\Z$ and $H$ is finite. Again by \Pref{prop:independent}, we may assume that $\mu$ is the uniform measure on $H$. By Gromov's theorem and the Bass--Guivarc'h formula, it follows that the growth of $G$ is at least quadratic, so in particular $\left|B_r\right|\ge r^{3/2}$ for sufficiently large $r$.

  Fix $\lambda>1$, and let $R_n=((L_n(g))_g,Z_n)$ denote the $\lambda$-stationary random walk on $H\wr G$. If $Z_n$ is transient, then $R_n$ is transient as well, so we assume that $Z_n$ is recurrent.

  We apply \Pref{prop:homesick-range-down} with $c=\frac{1}{2}$ to deduce that for sufficiently large $k$,
  \[
    \Pr\left(\left|\range(Z_{\rho_k})\right|\le \frac{N_k}{4}\right) \le e^{-\frac{N_k}{8}}
  \]
  for
  \[
    N_k = \min\set{\left|B_{\frac{\log k}{2\log\lambda}}\right|,\floor{\frac{\lambda-1}{\left|S\right|}\sqrt{k}}},
  \]
  so for sufficiently large $k$ we have
  \[
    N_k \ge \left(\frac{\log k}{2\log\lambda}\right)^{3/2} \eqqcolon C\log^{3/2}k.
  \]
  Applying \Pref{prop:ret-prob-gen}, we deduce that
  \begin{align*}
    \Pr(R_{\rho_k}=\id) & \le \Pr\left(R_{\rho_k}=\id\suchthat \left|\range(Z_{\rho_k})\right|\ge \frac{N_k}{4}\right) + \Pr\left(\left|\range(Z_{\rho_k})\right|\le \frac{N_k}{4}\right) \\
    & \le \left|H\right|^{-N_k/4} + e^{-N_k/8} \le \left|H\right|^{-C\log^{3/2}k/4} + e^{-C\log^{3/2}k/8},
  \end{align*}
  and thus $\sum_{k=1}^{\infty}\Pr(R_{\rho_k}=\id)<\infty$. By \Rref{rem:rec-tran-gen}, it follows that $R_n$ is transient, and this holds for every $\lambda>1$ as required.
\end{proof}

We finally prove:

\begin{proof}[Proof of \Tref{thm:no-transition}, $H$ infinite case]
  We assume now that $H$ is infinite, and thus there is a constant $C>0$ such that
  \begin{equation}\label{eq:ret-prob-H}
    \mu^{*2t}(\id_H) \le \frac{C}{\sqrt{t}}
  \end{equation}
  for all $t\ge 1$.

  For each $g\in S$, we write
  \[
    N_{g,k} = \left|\set{0\le i\le k-1\suchthat Z_{\rho_i+1}=g}\right|.
  \]
  Since the different excursions are independent, each $N_{g,k}$ is a binomial random variable with parameters $k$ and $\frac{1}{\left|S\right|}$. Writing $A=\set{\forall g\in S:N_{g,k}\ge\frac{k}{2\left|S\right|}}$, we may apply Chernoff's inequality and deduce
  \[
    \Pr\left(A^c\right) \le \left|S\right|e^{-\frac{k}{12\left|S\right|}}.
  \]

  Write $\overline{S}=S\cup\set{\id_G}$. \Pref{prop:ret-prob-gen} shows that
  \[
    \Pr(R_{\rho_k}=\id \suchthat Z_0,\dots,Z_{\rho_k}) \le \prod_{g\in\overline{S}}\mu^{*(2\xi_G(g,\rho_k))}(\id_H).
  \]
  We note that $\xi_G(s,\rho_k)\ge N_{s,k}$ for every $s\in S$, and that $\xi_G(\id_G,\rho_k)=k$; therefore, combining with \eqref{eq:ret-prob-H},
  \[
    \Pr(R_{\rho_k}=\id \suchthat A) \le \frac{C}{\sqrt{2k}}\left(\frac{C}{\sqrt{k/\left|S\right|}}\right)^{\left|\overline{S}\right|-1} = \frac{C'}{k^{\left|\overline{S}\right|/2}}.
  \]
  We can now deduce that
  \begin{align*}
    \Pr(R_{\rho_k}=\id) & \le \Pr(R_{\rho_k}=\id \suchthat A) + \Pr(A^c) \le \frac{C'}{k^{\left|\overline{S}\right|/2}} + \left|S\right|e^{-\frac{k}{12\left|S\right|}}.
  \end{align*}
  Since $G$ is infinite, it follows that $\left|\overline{S}\right|\ge 3$, hence $\sum_{k=1}^{\infty}\Pr(R_{\rho_k}=\id)<\infty$. This shows that $R_n$ is transient for every $\lambda>1$, completing the proof.
\end{proof}

\section{Stationary random walks on lamplighter graphs}\label{sec:graphs}

The notion of stationary random walks extends naturally to lamplighter graphs $\Gamma' \wr \Gamma$. Specifically, assume that $\Gamma$ and $\Gamma'$ are bounded-degree graphs, where~$\Gamma$ is infinite and $\Gamma'$ is nontrivial. We assume that $\Gamma$ is rooted, so the notion of homesick random walks on $\Gamma$ is defined with respect to the root $o$ of $\Gamma$. Then the ideas of the proofs of \Tref{thm:virt-Z} and \Tref{thm:no-transition} show:
\begin{enumerate}
  \item If the balls around $o$ in $\Gamma$ have linear growth and $\Gamma'$ is finite, then there is a phase transition as in \Tref{thm:virt-Z}.
  \item If the balls around $o$ in $\Gamma$ grow superlinearly, or if $\Gamma'$ is infinite, then the $\lambda$-stationary random walk is transient for every $\lambda>1$.
\end{enumerate}

In contrast to lamplighter groups, bounded-degree graphs can exhibit much wilder growth behaviours. It is therefore natural to expect that if the growth of balls in $\Gamma$ oscillates between linear and superlinear rates, the stationary random walk may display a phase transition, or even remain transient for every value of $\lambda$, depending on the oscillations.

Another contrast between stationary random walks on lamplighter graphs and on lamplighter groups arises in the critical case. While we conjecture that the walk on lamplighter groups is recurrent when the base group is virtually-$\Z$ (\Cjref{conj:conj}), we show that this fails for lamplighter graphs when the base group is not transitive.

\begin{exmpl}
  Let $\Gamma_m=(V_m,E_m)$ denote the graph obtained from $\Z$ by splitting every vertex of distance $2^i$ (for every $i\ge 1$) from $0$ to $m$ vertices. Formally, we write
  \[
    V_m = \left(\Z\setminus\set{\pm 2^i\suchthat i\ge 1}\right) \cup \set{(2^i,j)\suchthat i\ge 1,\;1\le j \le m}.
  \]
  The edges in $\Gamma_m$ are given as follows: each $x\in\Z$ is adjacent to $x-1,x+1$ (unless one of them is split), and each split vertex $(2^i,j)$ is adjacent to $2^i-1,2^i+1$. Then~$\Gamma_m$ is an infinite bounded-degree graph (with maximal degree $m+1$). We further note that
  \[
    \left|B_r(0)\right| = 2r + 2(m-1)\floor{\log_2 r}
  \]
  for every $r\ge 1$. We denote by $\pi\colon\Gamma_m\to\Z$ the natural projection of $\Gamma_m$ to $\Z$, given by $\pi(x)=x$ for every non-split $x\in V_m$ and $\pi(2^i,j)=2^i$.

  We consider the lamplighter graph $\Gamma=(\Z/2\Z)\wr\Gamma_m$. Since the growth of $\Gamma_m$ around $0$ is linear, there exists $1<\lambda_c<\infty$ such that the $\lambda$-stationary random walk is transient for $\lambda<\lambda_c$, and recurrent for $\lambda>\lambda_c$. To describe it explicitly, let $R_n=((L_n(v))_v,Z_n)$ denote the $\lambda$-stationary random walk on $\Gamma$. Then $Z_n$ is recurrent, and we denote by $(\rho_k)_k$ the successive return times to $0$.

  \begin{prop}
    Assume that $m\ge 3$. Then the $\lambda$-stationary random walk $R_n$ on~$\Gamma$ is recurrent for $\lambda>4$, and transient for $\lambda\le 4$.
  \end{prop}

  \begin{proof}
    Calculating effective resistances as in the proof of \Pref{prop:homesick-range-up} for $\Gamma_m$, we get an equality
    \[
        \Pr\left(\max_{0\le t\le\rho_1}\left|\pi(Z_t)\right|\ge r\right) = \frac{1}{2}\left(\sum_{s=0}^{r-1}\left|\partial_E B_s(0)\right|^{-1}\lambda^s\right)^{-1}.
    \]
    We note that
    \[
        \left|\partial_E B_s(0)\right| = \begin{cases} 2, & s \ne 2^j \text{ for } j\ge 1 \\ 2m, & s=2^j \text{ for }j\ge 1,\end{cases}
    \]
    implying that
    \[
        \frac{\lambda-1}{\lambda^r-1} \le \Pr\left(\max_{0\le t\le\rho_1}\left|\pi(Z_t)\right|\ge r\right) \le \frac{m(\lambda-1)}{\lambda^r-1}.
    \]

    We now prove the claim. Suppose first that $\lambda>4$. For each $k\ge\lambda+1$, since the different excursions are independent, we have
    \[
        \Pr\left(\max_{0\le t\le\rho_k}\left|\pi(Z_t)\right|<r\right) = \Pr\left(\max_{0\le t\le\rho_1}\left|\pi(Z_t)\right|<r\right)^k \ge \left(1-\frac{\lambda-1}{\lambda^r-1}\right)^k.
    \]
    for every $r\ge 1$. Taking $r_k\coloneqq\ceil{\frac{\log k}{\log\lambda}}$, we deduce
    \[
        \Pr\left(\max_{0\le i\le\rho_k}\left|\pi(Z_i)\right|<r_k\right) \ge \left(1-\frac{\lambda-1}{k-1}\right)^k \ge \frac{1}{\lambda^{\lambda+1}}
    \]
    where the last inequality follows from $k\ge\lambda+1$ and monotonicity. We note that if $\max_{0\le t\le\rho_k}\left|\pi(Z_t)\right|<r_k$ then $\range(Z_{\rho_k}) \sub B_{r_k}(0)$, and thus
    \[
        \Pr\left(\left|\range(Z_{\rho_k})\right|\le \left|B_{r_k}(0)\right|\right) \ge \frac{1}{\lambda^{\lambda+1}}.
    \]
    It follows (using \Pref{prop:ret-prob-gen} for graphs) that
    \begin{align*}
      \Pr\left(R_{\rho_k}=\id\right) & \ge \Pr\left(\left|\range(Z_{\rho_k})\right|\le \left|B_{r_k}(0)\right|\right)\Pr\left(R_{\rho_k}=\id \suchthat \left|\range(Z_{\rho_k})\right|\le \left|B_{r_k}(0)\right|\right) \\
      & \ge \frac{1}{\lambda^{\lambda+1}} \cdot 2^{-\left|B_{r_k}(0)\right|} \ge \frac{1}{2\lambda^{\lambda+1}}\cdot \frac{1}{k^{2\log 2/\log\lambda}(\log k)^{2(m-1)}}.
    \end{align*}
    Therefore $\sum_{k=1}^{\infty}\Pr\left(R_{\rho_k}=\id\right)=\infty$ for $\lambda>4$, showing that $R_n$ is recurrent in this case.

    We now suppose that $\lambda\le 4$. Fix $k\ge 3$. We set
    \[
        r_k'\coloneqq\floor{\frac{\log k-\log\log k+\log((\lambda-1)/4)}{\log\lambda}},
    \]
    and define the events
    \[
        A^+_k \coloneqq \set{\max_{0\le t\le\rho_k} \pi(Z_t) \ge r_k'}
    \]
    and
    \[
        A^-_k \coloneqq \set{\min_{0\le t\le\rho_k} \pi(Z_t) \le -r_k'}.
    \]
    At each excursion we have
    \[
        \Pr\left(\max_{0\le t\le\rho_1} \pi(Z_t) \ge r_k'\right) = \frac{1}{2}\Pr\left(\max_{0\le t\le\rho_1} \left|\pi(Z_t)\right| \ge r_k'\right) \ge \frac{\lambda-1}{2\lambda^{r_k'}} \ge \frac{2\log k}{k}.
    \]
    Therefore, since the different excursions are independent,
    \[
        \Pr((A^+_k)^c) = \Pr\left(\max_{0\le t\le\rho_1} \pi(Z_t) < r_k'\right)^k \le \left(1 - \frac{2\log k}{k}\right)^k \le \frac{1}{k^2}.
    \]
    By symmetry, $\Pr((A^-_k)^c)\le\frac{1}{k^2}$ as well.

    Next, we write $r_k''=\floor{\frac{\log k}{2\log\lambda}}$. For each $1\le i\le \log_2 r_k''$ and $1\le j\le m$, we write~$B_{i,j,k}^+$ for the event that $Z_n$ visits $(2^i,j)$ until time $\rho_k$, and $B_{i,j,k}^-$ for the same event for $(-2^i,j)$. At a given excursion we have
    \[
        \Pr\left(\exists t\le\rho_1:Z_t=(2^i,j)\right) \ge \frac{1}{m}\Pr\left(\max_{0\le t\le \rho_1}\pi(Z_t)\ge r_k''\right) \ge \frac{1}{2m}\frac{\lambda-1}{\lambda^{r_k''}-1} \ge \frac{\lambda-1}{2mk^{1/2}}.
    \]
    Therefore
    \[
        \Pr((B_{i,j,k}^+)^c) \le \left(1 - \frac{\lambda-1}{2mk^{1/2}}\right)^k \le e^{-\frac{\lambda-1}{2m}k^{1/2}},
    \]
    and by symmetry
    \[
        \Pr((B_{i,j,k}^-)^c) \le \left(1 - \frac{\lambda-1}{2mk^{1/2}}\right)^k \le e^{-\frac{\lambda-1}{2m}k^{1/2}}.
    \]

    We note that under the event $\tilde{A}\coloneqq A_k^+\cap A_k^-\cap\bigcap_{1\le i\le\log_2 r_k'',1\le j\le m}(B_{i,j,k}^+\cap B_{i,j,k}^-)$ we have
    \[
        \left|\range(Z_{\rho_k})\right| \ge 2r_k' + 2(m-1)\log_2 r_k'' \ge \frac{2\log k}{\log\lambda} + (2m-3)\log_2 \log k.
    \]
    We apply \Pref{prop:ret-prob-gen} for graphs again, and deduce
    \begin{align*}
      \Pr(R_{\rho_k}=\id) & \le \Pr\left(R_{\rho_k}=\id\suchthat\tilde{A}\right) + \Pr((A_k^+)^c) + \Pr((A_k^-)^c) + \sum_{\substack{1\le i\le\log_2 r_k''\\1\le j\le m}}\Pr(B_{i,j,k}^c) \\
      & \le 2^{-(\frac{2\log k}{\log\lambda} + (2m-3)\log_2 \log k)} + \frac{2}{k^2} + m\log_2 r_k''e^{-\frac{\lambda-1}{2m}k^{1/2}} \\
      & \le \frac{1}{k^{-\frac{\log 4}{\log\lambda}}(\log k)^{2m-3}} + \frac{2}{k^2} + m\log_2 \log k e^{-\frac{\lambda-1}{2m}k^{1/2}},
    \end{align*}
    which is summable for $m\ge 3$ given $\lambda\le 4$. Hence $R_n$ is transient for $\lambda\le 4$, concluding the proof.

  \end{proof}
\end{exmpl}

\bibliographystyle{plain}
\bibliography{refs}

\end{document}